\documentclass[12pt]{amsart}
\usepackage{amssymb,amsmath,amsfonts,latexsym}
\usepackage{bm}

\newcommand{\newsection}[1]{\setcounter{equation}{0} \section{#1}}
\newcommand{\bea}{\begin{eqnarray}}
\newcommand{\eea}{\end{eqnarray}}

\newcommand{\vp}{\varphi}

\newcommand{\clb}{\mathcal{B}}

\newcommand{\cld}{\mathcal{D}}
\newcommand{\cle}{\mathcal{E}}
\newcommand{\clf}{\mathcal{F}}

\newcommand{\clh}{\mathcal{H}}
\newcommand{\clk}{\mathcal{K}}
\newcommand{\cll}{\mathcal{L}}
\newcommand{\clm}{\mathcal{M}}

\newcommand{\clq}{\mathcal{Q}}

\newcommand{\cls}{\mathcal{S}}

\newcommand{\D}{\mathbb{D}}

\newcommand{\raro}{\rightarrow}

\def\textmatrix#1&#2\\#3&#4\\{\bigl({#1 \atop #3}\ {#2 \atop #4}\bigr)}
\def\dispmatrix#1&#2\\#3&#4\\{\left({#1 \atop #3}\ {#2 \atop #4}\right)}
\newcommand{\be}{\begin{equation}}
\newcommand{\ee}{\end{equation}}
\newcommand{\ben}{\begin{eqnarray*}}
\newcommand{\een}{\end{eqnarray*}}

\newcommand{\NI}{\noindent}

\newcommand{\bi}{\begin{itemize}}
\newcommand{\ei}{\end{itemize}}

\newcommand{\ci}{C_{z_i}}
\newcommand{\cj}{C_{z_j}}
\newcommand{\mi}{M_{z_i}}
\newcommand{\mj}{M_{z_j}}
\newcommand{\he}{H^2_{\mathcal{E}}(\mathbb{D}^n)}
\newcommand{\pq}{P_{\mathcal{Q}}}
\newcommand{\ps}{P_{\mathcal{S}}}
\newcommand{\Z}{\mathbb{Z}}

\newtheorem{Question}{\sc Question}

\theoremstyle{definition}

\theoremstyle{plain}

\newtheorem{thm}{Theorem}[section]
\newtheorem{cor}[thm]{Corollary}
\newtheorem{lem}[thm]{Lemma}

\theoremstyle{definition}

\newtheorem{ex}[thm]{Example}

\numberwithin{equation}{section}

\let\phi=\varphi

\begin{document}

\title[Beurling quotient modules]{Beurling quotient modules on the polydisc}

\author[Bhattacharjee]{Monojit Bhattacharjee}
\address{Department of Mathematics, Indian Institute of Technology Bombay, Powai, Mumbai, 400076,
India}
\email{mono@math.iitb.ac.in, monojit.hcu@gmail.com}

\author[Das] {B. Krishna Das}
\address{Department of Mathematics, Indian Institute of Technology Bombay, Powai, Mumbai, 400076,
India}
\email{dasb@math.iitb.ac.in, bata436@gmail.com}

\author[Debnath]{Ramlal Debnath}
\address{Indian Statistical Institute, Statistics and Mathematics Unit, 8th Mile, Mysore Road, Bangalore, 560059,
India}
\email{ramlal-rs@isibang.ac.in, ramlaldebnath100@gmail.com }

\author[Sarkar]{Jaydeb Sarkar}
\address{Indian Statistical Institute, Statistics and Mathematics Unit, 8th Mile, Mysore Road, Bangalore, 560059,
India}
\email{jay@isibang.ac.in, jaydeb@gmail.com}




\begin{abstract}
Let $H^2(\mathbb{D}^n)$ denote the Hardy space over the polydisc $\mathbb{D}^n$, $n \geq 2$. A closed subspace $\mathcal{Q} \subseteq H^2(\mathbb{D}^n)$ is called Beurling quotient module if there exists an inner function $\theta \in H^\infty(\mathbb{D}^n)$ such that $\mathcal{Q} = H^2(\mathbb{D}^n) /\theta H^2(\mathbb{D}^n)$. We present a complete characterization of Beurling quotient modules of $H^2(\mathbb{D}^n)$: Let $\mathcal{Q} \subseteq H^2(\mathbb{D}^n)$ be a closed subspace, and let $C_{z_i} = P_{\mathcal{Q}} M_{z_i}|_{\mathcal{Q}}$, $i=1, \ldots, n$. Then $\mathcal{Q}$ is a Beurling quotient module if and only if
\[
(I_{\mathcal{Q}} - C_{z_i}^* C_{z_i}) (I_{\mathcal{Q}} - C_{z_j}^* C_{z_j}) = 0 \qquad (i \neq j).
\]
We present two applications: first, we obtain a dilation theorem for Brehmer $n$-tuples of commuting contractions, and, second, we relate joint invariant subspaces with factorizations of inner functions. All results work equally well for general vector-valued Hardy spaces.
\end{abstract}

\maketitle

\tableofcontents

\newsection{Introduction}

Let $n \geq 1$ and let $\cle$ be a Hilbert space. The $\cle$-valued Hardy space over the polydisc $\D^n$, denoted by $H^2_{\cle}(\D^n)$, is the Hilbert space of all $\cle$-valued analytic functions $f$ on $\D^n$ such that
\[
\|f\| := \Big(\sup_{0 \leq r < 1} \int_{\mathbb{T}^n} \|f(rz_1, \ldots, r z_n)\|_{\cle}^2 d{m}({z}) \Big)^{\frac{1}{2}} < \infty,
\]
where $z = (z_1, \ldots, z_n)$ and $dm(z)$ is the normalized Lebesgue measure on the $n$-torus $\mathbb{T}^n$.  Given another Hilbert space $\cle_*$, we denote by $H^{\infty}_{\clb(\cle,\cle_*)}(\D^n)$ the Banach space of all $\clb(\cle, \cle_*)$-valued bounded analytic functions on $\D^n$. A function $\Theta \in H^{\infty}_{\clb(\cle,\cle_*)}(\D^n)$ is called \textit{inner} if $f \mapsto \Theta f$ defines an isometry $M_{\Theta}: H^2_{\cle_*}(\D^n) \raro \he$. The simplest example is $\Theta(z) = z_i I_{\cle}$, $i=1, \ldots, n$, whenever $\cle_* = \cle$. Therefore, $(M_{z_1}, \ldots, M_{z_n})$ is a commuting tuple of isometries on $\he$. Let $\cll \subseteq \he$ be a closed subspace. Then $\cll$ is said to be a \textit{quotient module} if $M_{z_i}^* \cll \subseteq \cll$ for all $i=1, \ldots, n$. The subspace $\cll$ is called a \textit{submodule} if $\cll^\perp$ is a quotient module \cite{DP}.

We pause for a brief aside to remark that if $n=1$, then a closed subspace $\clq \subseteq H^2_{\cle}(\D)$ is a quotient module if and only if there exist a Hilbert space $\cle_*$ and an inner function $\Theta \in H^\infty_{\clb(\cle_*, \cle)}(\D)$ such that $\clq^\perp = \Theta H^2_{\cle_*}(\D)$, or equivalently
\[
\clq = H^2_{\cle}(\D) \ominus \Theta H^2_{\cle_*}(\D) \cong H^2_{\cle}(\D) /\Theta H^2_{\cle_*}(\D).
\]
This follows from the classical Beurling-Lax-Halmos theorem \cite{NF}. Therefore, the preceding statement gives a satisfactory description of quotient modules of $H^2_{\cle}(\D)$. It is also worthwhile to emphasize that there is an inseparable alliance between quotient modules and bounded linear operators on Hilbert spaces. For instance, if $\clq$ is a quotient module of $H^2_{\cle}(\D)$, then the \textit{module operator} (also known as \textit{model operator}) $M_{\clq} :=P_{\clq} M_z|_{\clq}$ is a pure contraction on $\clq$. The classical Sz.-Nagy and Foias theory says that, up to unitary equivalence, these are all pure contractions on Hilbert spaces. Recall that an operator $X \in \clb(\clh)$ is called (i) a \textit{contraction} if $I_{\clh} - X X^* \geq 0$, or, equivalently, $\|Xh\|_{\clh} \leq \|h\|_{\clh}$ for all $h \in \clh$, and (ii) \textit{pure} if the sequence $\{X^{*m}\}_{m \geq 0}$ converges to $0$ in the strong operator topology.

Therefore, quotient modules of $\he$, $n \geq 1$, are of interest in operator theory, function theory, and operator algebras. However, in sharp contrast, the situation changes dramatically in the case when $n > 1$: In general, a quotient module of $\he$ does not necessarily admit a Beurling-type representation. In fact, concrete description of quotient modules of $\he$ is commonly regarded as one of the most difficult and important problems in modern operator theory and function theory \cite{CG, DP, R, Yang}.

In this paper, our interest is in comparing the variability of the classical Beurling representations of quotient modules in several variables. For a Hilbert space $\cle$ and a closed subspace $\clq \subseteq \he$, we say that $\clq$ is a \textit{Beurling quotient module} (and $\clq^\perp$ is a \textit{Beurling submodule}) if
\[
\clq = \he \ominus \Theta H^2_{\cle_*}(\D^n) \cong \he/\Theta H^2_{\cle_*}(\D^n),
\]
for some Hilbert space $\cle_*$ and inner function $\Theta \in H^\infty_{\clb(\cle_*, \cle)}(\D^n)$. Since $\mi M_{\Theta} = M_{\Theta} \mi$ for all $i=1, \ldots, n$, it follows, in particular, that $\clq$ ($\clq^\perp$) is also a quotient module (submodule) of $\he$. In the context of the above discussion, it appears natural to raise the following question:

\begin{Question}\label{Q: 1}
Which quotient modules of $\he$ admit Beurling representations?
\end{Question}

Curiously, despite its natural appeal and all possible applications, the above question remained fairly untouched. It is also the one variable work of Beurling \cite{B} which stirred our interest in this question. Evidently, this has a lot to do with the module (or model) operators associated with quotient modules. Given a quotient module $\clq \subseteq \he$, define the $n$-tuple of commuting contractions $C_z=(C_{z_1}, \ldots, C_{z_n})$ (we call it the \textit{tuple of module operators} or \textit{module operators} in short) on $\clq$ by
\[
\ci = P_{\clq} M_{z_i}|_{\clq} \qquad (i=1,\ldots, n),
\]
where $P_{\clq}$ is the orthogonal projection from $\he$ onto $\clq$. Therefore, $\clq$ is a contractive Hilbert module over $\mathbb{C}[z_1,\ldots, z_n]$ in the following sense (see \cite{CG, DP}):
\[
(p,h) \in \mathbb{C}[z_1,\ldots, z_n] \times \clq \longrightarrow p(C_{z_1},\ldots, C_{z_n})h \in \clq.
\]
The following theorem provides the answer to Question \ref{Q: 1}:

\begin{thm}\label{Char dc in terms of compre}
Let $\cle$ be a Hilbert space and let $\clq$ be a quotient module of $H^2_{\cle}(\D^n)$. Then $\clq$ is a Beurling quotient module if and only if
\[
(I_{\clq}- \ci^* \ci)(I_{\clq}-\cj^* \cj) = 0 \qquad (i \neq j).
\]
\end{thm}

The proof of Theorem \ref{Char dc in terms of compre} depends on several lemmas, some of which are of independent interest and related to the delicate structure of submodules and quotient modules of (vector-valued) Hardy space over $\D^n$. This is the content of Section \ref{sec: main theorem}.

In section \ref{sec: dilations}, we apply the above framework to dilations of $n$-tuples of commuting contractions. Let us explain this when $n=2$. A pair of commuting contractions $(T_1, T_2)$ on $\clh$ is called \textit{Brehmer} if $T_1$ and $T_2$ are pure and $D \geq 0$, where
\[
D := I_{\clh} - T_1 T_1^* - T_2 T_2^* + T_1 T_2 T_1^* T_2^*.
\]
It is known \cite{CV, MV} that a Brehmer pair dilates to $(M_{z_1}, M_{z_2})$ on a vector-valued Hardy space, or, equivalently, $(M_{z_1}, M_{z_2})$ on vector-valued Hardy spaces are analytic models of Brehmer pair. More specifically, if $(T_1, T_2)$ is a Brehmer pair, then there exist a Hilbert space $\cld$ (which is actually $\overline{\mbox{ran}} D^{\frac{1}{2}}$) and a quotient module $\clq \subseteq H^2_{\cld}(\D^2)$ such that
\[
(T_1, T_2) \cong (P_{\clq} M_{z_1}|_{\clq}, P_{\clq} M_{z_2}|_{\clq}).
\]
Since $\clq$ is not necessarily a Beurling quotient module, this model is not completely comparable with the classical Sz.-Nagy and Foias analytic models of pure contractions. The missing piece is precisely a paraphrase of Theorem \ref{Char dc in terms of compre}: Let $(T_1, T_2)$ be a pair of commuting contractions on $\clh$, and let $\cld = \overline{ran} D^{\frac{1}{2}}$ . Then there exist a Hilbert space $\cle_*$ and an inner function $\Theta \in H^\infty_{\clb(\cle_*, \cld)}(\D^n)$ such that
\[
(T_1, T_2) \cong (P_{\clq_\Theta} M_{z_1}|_{\clq_{\Theta}}, P_{\clq_\Theta} M_{z_2}|_{\clq_\Theta}),
\]
if and only if the pair $(T_1, T_2)$ is Brehmer and
\[
(I_{\clh}-T_1^*T_1)(I_{\clh}-T_2^*T_2)=0.
\]
Here $\clq_\Theta := H^2_{\cld}(\D^n)/\Theta H^2_{\cle_*}(\D^n)$ is the Beurling quotient module of $H^2_{\cld}(\D^2)$ corresponding to the inner function $\Theta \in H^\infty_{\clb(\cle_*, \cld)}(\D^n)$. This is the main content of Theorem \ref{characteristic function}.

Section \ref{sect: inv sub} deals with factorizations of inner functions in $H^\infty_{\clb(\cle_*, \cle)}(\D^n)$ and invariant subspaces of tuples of module operators. We briefly explain the main content of Section \ref{sect: inv sub} when $\cle = \mathbb{C}$ and $n=2$. The starting point is the following one variable result (see Sz.-Nagy and Foias, and Bercovici \cite{HB,NF}), which connects invariant subspaces of module operators with factorizations of the corresponding inner functions:

\NI Let $\theta\in H^\infty(\D)$ be an inner function. Then $T_{\theta} := P_{\clq_{\theta}} M_z|_{\clq_{\theta}}$ has an invariant subspace if and only if there exist inner functions $\vp$ and $\psi$ in $H^\infty(\D)$ such that
\[
\theta =\vp \psi.
\]
However, in the case of $H^\infty(\D^2)$, the existence of joint invariant subspaces is not sufficient to ensure factorizations of inner functions (see Example \ref{example: not factorable}). Theorem \ref{inv subsp in D^n} deals with this missing link: Let $\theta\in H^\infty(\D^2)$ be an inner function, $\clq_{\theta} = H^2(\D^2)/\theta H^2(\D^2)$, and let $T_\theta = (P_{\clq_{\theta}} M_{z_1}|_{\clq_{\theta}}, P_{\clq_{\theta}} M_{z_2}|_{\clq_{\theta}})$ denote the pair of module operators. The following are equivalent.
\begin{enumerate}
\item $\theta = \vp \psi$ for some inner functions $\vp, \psi \in H^\infty(\D^2)$.

\item There exists a joint $T_{\theta}$-invariant subspace $\clm \subseteq \clq_{\theta}$ such that $\clm\oplus {\theta}H^2(\D^2)$ is a Beurling submodule of $H^2(\D^2)$.

\item There exists a joint $T_{\theta}$-invariant subspace $\clm \subseteq \clq_{\theta}$ such that 
    \[
    (I-C_1^*C_1)(I-C_2^*C_2)=0,
    \]
where $C_i = P_{\clq_{\theta}\ominus \clm}M_{z_i}|_{\clq_{\theta}\ominus\clm}$ and $i = 1,2$.
\end{enumerate}

In Corollary \ref{cor: unique}, we prove that nontrivial factorizations is equivalent to the existence of nontrivial invariant subspaces of tuples of module operators.

All the Hilbert spaces considered in this paper are assumed to be complex and separable. Given Hilbert spaces $\cle$ and $\cle_*$, we denote by $\clb(\cle_*, \cle)$ (or simply by $\clb(\cle)$ if $\cle_* = \cle$) the Banach space of all bounded linear operators from $\cle_*$ to $\cle$. We say that two $n$-tuples $T=(T_1, \ldots, T_n)$ on $\clh$ and $R = (R_1, \ldots, R_n)$ on $\clk$ are unitarily equivalent (which we denote by $T \cong R$) if there exists a unitary $U \in \clb(\clh, \clk)$ such that $UT_i = R_iU$ for all $i=1, \ldots, n$.

Finally, it is worth pointing out that the general definition of a Brehmer pair $(T_1, T_2)$ (or an $n$-tuple) does not assume that the $T_i$'s are pure. Here the restricted definition fits well with the analytical model framework for our class of operators.

\newsection{Proof of Theorem \ref{Char dc in terms of compre}}\label{sec: main theorem}

Throughout this section we fix a Hilbert space $\cle$ and a quotient module $\clq$ of $\he$. We denote by $\cls$ the submodule $\clq^\perp$, that is
\[
\cls = \he \ominus \clq \cong \he/\clq.
\]
In order to shorten some of our computations, we will use the standard notation of cross-commutators: $[T_1,T_2]:= T_1 T_2 - T_2 T_1$ whenever $T_1$ and $T_2$ are bounded linear operators on some Hilbert space.

Now, by definition, $\clq$ is a Beurling quotient module if and only if there exist a Hilbert space $\cle_*$ and an inner function $\Theta \in H^\infty_{\clb(\cle_*, \cle)}(\D^n)$ such that $\cls = \Theta H^2_{\cle_*}(\D^n)$, which, by \cite{Mand, SSW}, equivalent to the condition that $[R_{z_j}^*, R_{z_i}] = 0$ for all $i\neq j$, where $R_{z_r} = M_{z_r}|_{\cls}$, and $r=1,\ldots,n$. Then we have the following interpretation of Theorem \ref{Char dc in terms of compre}:

\begin{lem}\label{lemma: reduction}
For each $i$ and  $j$ in $\{1, \ldots,n\}$, define
\[
X_{ij} = P_{\cls} \mi P_{\clq} \mj^* P_{\cls}.
\]
Then $\clq$ is a Beurling quotient module if and only if $X_{ij} = 0$ for all $i \neq j$.
\end{lem}
\begin{proof}
Suppose $i \neq j$. Since $M_{z_i}M_{z_j}^* = M_{z_j}^*M_{z_i}$ and $I_{\he} - \ps = \pq$, it follows that
\[
[R_{z_j}^*, R_{z_i}] = R_{z_j}^{*}R_{z_i} - R_{z_i} R_{z_j}^{*} = P_{\cls} M_{z_{j}}^{*}M_{z_{i}}|_{\cls} - \ps M_{z_{i}}P_{\cls}M_{z_{j}}^{*}|_{\cls} = P_{\cls}M_{z_{i}}P_{\mathcal{Q}}M_{z_{j}}^{*}|_{\cls}.
\]
Therefore, $[R_{z_j}^*, R_{z_i}]=0$ if and only if $(P_{\cls}M_{z_{i}}P_{\mathcal{Q}}M_{z_{j}}^{*})|_{\cls} = 0$, which is equivalent to $X_{ij} = (P_{\cls}M_{z_{i}}P_{\mathcal{Q}}M_{z_{j}}^{*})P_{\cls} = 0$.
\end{proof}

It is often convenient to work with $P_{\clq} M_{z_i} P_{\clq} \in \clb(\he)$, which we will denote by $C_i$, that is
\[
C_i = P_{\clq} M_{z_i} P_{\clq} \in \clb(\he) \qquad (i=1, \ldots, n).
\]
Observe that $C = (P_{\clq} M_{z_1} P_{\clq}, \ldots, P_{\clq} M_{z_n} P_{\clq})$ is an $n$-tuple of commuting contractions on $\he$ (or, equivalently, $C$ defines a contractive $\mathbb{C}[z_1, \ldots, z_n]$-Hilbert module structure on $\he$), and 
\[
C_i|_{\clq} = \ci \text{ and } C_i^*|_{\clq} = \ci^*,
\]
for all $i=1, \ldots, n$. Finally, to shorten notation we set $T^k = T_1^{k_1} \cdots T_n^{k_n}$ whenever $T = (T_1, \ldots, T_n)$ is a commuting tuple on some Hilbert space and $k = (k_1, \ldots, k_n) \in \Z_+^n$.

\textsf{For the rest of this section, we fix $i$ and $j$ from $\{1, \ldots, n\}$, and assume that $i \neq j$.} In what follows, we will use $\hat{k}_i$ ($\hat{k}_j$) to denote multi-indices in $\Z_+^n$ whose $i$-th ($j$-th) slot has zero entry. The following lemma will play a key role.

\begin{lem}\label{lemma: [Ci, C ki]}
$[C_i, C^{*\hat{k}_i}] = \pq M_z^{*\hat{k}_i} \ps M_{z_i} \pq$ for all $\hat{k}_i \in \Z_+^n \setminus \{0\}$.
\end{lem}
\begin{proof}
First notice that $C^{*l} = M_{z}^{*l} \pq$ and $C^{l} = \pq M_{z}^{l}$ for all $l \in \Z_+^n$. Since $[C_i, C^{*\hat{k}_i}] = C_i C^{*\hat{k}_i} - C^{*\hat{k}_i} C_i$, it follows that
\[
[C_i, C^{*\hat{k}_i}]  = \pq M_{z_i} M_{z}^{*\hat{k}_i} \pq - \pq M_{z}^{*\hat{k}_i} \pq M_{z_i} \pq.
\]
Then, writing $\pq = I_{\he} - \ps$ into the middle of the second term on the right side and using $M_{z}^{*\hat{k}_i} M_{z_i} = M_{z_i} M_{z}^{*\hat{k}_i}$, we get the desired equality.
\end{proof}

For each $t = 1, \ldots, n$, we set $D_{C_t} = (\pq - C_{t}^* C_{t})^{\frac{1}{2}}$. Since
\[
C_{t}^* C_{t} =\pq M_{z_t}^* \pq M_{z_t}\pq = \pq M_{z_t}^* (I_{\he} - \ps) M_{z_t}\pq = \pq - \pq M_{z_t}^* \ps M_{z_t}\pq,
\]
that $D_{C_{t}}$ is well defined follows from the fact that
\begin{equation}\label{eqn: P_Q - CC*}
\pq - C_{t}^* C_{t} = \pq M_{z_t}^* \ps M_{z_t}\pq \geq 0.
\end{equation}
We now recall a classical result due to R. Douglas \cite{Douglas}. Let $A$ and $B$ be contractions on a Hilbert space $\clh$. The Douglas's range and inclusion theorem then says that $AA^* \leq BB^*$ if and only if there exists a contraction $X$ such that $A=BX$. We are now ready for the third key lemma of this section.

\begin{lem}\label{lemma: factor}
Suppose $\hat{k}_i \in \Z_+^n \setminus \{0\}$. There exist contractions $X_{\hat{k}_i}$ and $Y_{\hat{k}_i}$ in $\clb(\clq)$ such that
\[
[C_i, C^{*\hat{k}_i}] = X_{\hat{k}_i} D_{C_i} \mbox{~and~} [C^{\hat{k}_i}, C_i^*] = D_{C_i} Y_{\hat{k}_i}.
\]
\end{lem}
\begin{proof}
We already know that $D_{C_i}^2 = \pq - C_{t}^* C_{t} = \pq M_{z_i}^* \ps M_{z_i}\pq$. Then, by Lemma \ref{lemma: [Ci, C ki]}, we have
\[
\begin{split}
D_{C_i}^2 - [C_i, C^{*\hat{k}_i}]^* [C_i, C^{*\hat{k}_i}] & = \pq M_{z_i}^* \ps M_{z_i}\pq - (\pq M_{z_i}^* \ps M_z^{\hat{k}_i}) \pq (M_z^{*\hat{k}_i} \ps M_{z_i} \pq)
\\
& = \pq M_{z_i}^* \ps (I_{\he} -  M_z^{\hat{k}_i} \pq  M_z^{*\hat{k}_i})\ps M_{z_i} \pq
\\
& = (\pq M_{z_i}^* \ps) (I_{\he} -  M_z^{\hat{k}_i} \pq  M_z^{*\hat{k}_i}) (\pq M_{z_i}^* \ps)^*.
\end{split}
\]
Since $M_z^{\hat{k}_i} \pq$ is a contraction, it follows that $I_{\he} -  M_z^{\hat{k}_i} \pq  M_z^{*\hat{k}_i} \geq 0$, and hence
\[
D_{C_i}^2 - [C_i, C^{*\hat{k}_i}]^* [C_i, C^{*\hat{k}_i}] \geq 0.
\]
Then the first equality is an immediate consequence of the Douglas's range and inclusion theorem. Finally, since $[C_i, C^{*\hat{k}_i}]^* = [C^{\hat{k}_i}, C_i^*]$, the second equality follows from the first.
\end{proof}

The final ingredient is the following result. Again recall that $X_{ij} = \ps M_{z_i} \pq M_{z_j}^* \ps$ (see Lemma \ref{lemma: reduction}).

\begin{lem}\label{lemma: 3 zero}
If $D_{C_i} D_{C_j} = 0$, then, for each $\hat{k}_i, \hat{l}_j \in \Z_+^n \setminus \{0\}$,
\begin{enumerate}
\item $\pq M_{z}^{*\hat{k}_i} X_{ij} M_{z}^{\hat{l}_j} \pq = 0$,
\item $\pq M_{z_i}^* X_{ij} M_{z}^{\hat{l}_j} \pq = 0$, and
\item $\pq M_{z}^{*\hat{k}_i} X_{ij} M_{z_j} \pq = 0$.
\end{enumerate}
\end{lem}
\begin{proof}
By Lemma \ref{lemma: factor}, we have on one hand $[C_i, C^{*\hat{k}_i}] [C_j, C^{*\hat{l}_j}]^* = 0$, and on the other hand, by Lemma \ref{lemma: [Ci, C ki]},
\[
\begin{split}
[C_i, C^{*\hat{k}_i}] [C_j, C^{*\hat{l}_j}]^* & = (\pq M_z^{*\hat{k}_i} \ps M_{z_i}) \pq (M_{z_j}^* \ps M_z^{\hat{l}_j} \pq)
\\
& = \pq M_z^{*\hat{k}_i} X_{ij} M_z^{\hat{l}_j} \pq.
\end{split}
\]
This proves $(1)$. To verify $(2)$, first observe that \eqref{eqn: P_Q - CC*} implies
\[
(\pq - C_i^* C_i)[C^{\hat{l}_j}, C_j^*] = (\pq M_{z_i}^* \ps M_{z_i}\pq) [C_j, C^{*\hat{l}_j}]^*.
\]
By Lemma \ref{lemma: [Ci, C ki]}, we can write $[C_j, C^{*\hat{l}_j}]^* = \pq M_{z_j}^* \ps M_z^{\hat{l}_j} \pq$, where, on the other hand, Lemma \ref{lemma: factor} implies that 
\[
(\pq - C_i^* C_i)[C^{\hat{l}_j}, C_j^*] = 0.
\]
Therefore
\[
0 = (\pq M_{z_i}^* \ps M_{z_i}\pq) (\pq M_{z_j}^* \ps M_z^{\hat{l}_j} \pq) = \pq M_{z_i}^* (\ps M_{z_i}\pq M_{z_j}^* \ps) M_z^{\hat{l}_j} \pq,
\]
which proves $(2)$. The proof of $(3)$ is similar to that of $(2)$: We first observe that
\[
[C_i, C^{*\hat{k}_i}] (\pq - C_j^* C_j) = \pq M_{z}^{*\hat{k}_i} X_{ij} M_{z_j} \pq
\]
whereas Lemma \ref{lemma: factor} implies that $[C_i, C^{*\hat{k}_i}] (\pq - C_j^* C_j) = 0$.
\end{proof}

We also need the following simple observation: $\clq$ reduces $(M_{z_1}^*P_{\cls}M_{z_1}, \ldots, M_{z_n}^*P_{\cls}M_{z_n})$, that is
\begin{equation}\label{eqn: Q reduces}
\pq (M_{z_t}^*P_{\cls}M_{z_t}) = (M_{z_t}^*P_{\cls}M_{z_t}) \pq \qquad (t=1,\ldots,n).
\end{equation}
Indeed, for a fixed $t$ in $\{1,\ldots, n\}$, writing $\pq = I_{\he} - \ps$, we see that
\[
\pq (M_{z_t}^*P_{\cls}M_{z_t}) \pq = M_{z_t}^*P_{\cls}M_{z_t} \pq - \ps M_{z_t}^*P_{\cls}M_{z_t} \pq,
\]
and, on the other hand, $\ps M_{z_t}^*P_{\cls} = \ps M_{z_t}^*$ and $M_{z_t}^* M_{z_t} = I_{\he}$ implies that
\[
\ps M_{z_t}^*P_{\cls}M_{z_t} \pq = \ps M_{z_t}^* M_{z_t} \pq = \ps \pq = 0.
\]
That is, $\pq (M_{z_t}^*P_{\cls}M_{z_t}) \pq = (M_{z_t}^*P_{\cls}M_{z_t}) \pq$. Then the claim follows from the fact that $M_{z_t}^*P_{\cls}M_{z_t}$ is a self-adjoint operator.

Now we are ready to plunge into the main body of the proof of Theorem \ref{Char dc in terms of compre}.

\smallskip

\begin{proof}[Proof of Theorem \ref{Char dc in terms of compre}]
Suppose $\clq$ is a Beurling quotient module. Then there exist a Hilbert space $\cle_*$ and an inner function $\Theta\in H^{\infty}_{\clb(\cle_*,\cle)}(\D^n)$ such that $\cls=\Theta H^2_{\clf}(\D^n)$ (see the discussion preceding Lemma \ref{lemma: reduction}). Then $\ps = M_{\Theta} M_{\Theta}^*$. Now \eqref{eqn: P_Q - CC*} and \eqref{eqn: Q reduces} implies that
\[
P_{\clq}-C_t^*C_t = M_{z_t}^*P_{\cls}M_{z_t} P_{\clq} \qquad (t=1,\ldots,n).
\]
Therefore by applying \eqref{eqn: Q reduces} again we obtain
\[
(P_{\clq}-C_i^*C_i)(P_{\clq}-C_j^*C_j)= (M_{z_i}^*P_{\cls}M_{z_i}) \pq (M_{z_j}^*P_{\cls}M_{z_j})P_{\clq} = M_{z_i}^*P_{\cls}M_{z_i} M_{z_j}^*P_{\cls}M_{z_j} P_{\clq}.
\]
We know by $M_{\Theta}^* M_{\Theta} = I_{\he}$ and $M_t M_{\Theta} = M_{\Theta} M_t$ for all $t = 1, \ldots, n$, that $M_{\Theta}^* \mj^* \mi M_{\Theta} = \mj^* \mi$. Then $\ps = M_{\Theta} M_{\Theta}^*$ implies that $\ps \mj^* \mi \ps  = M_{\Theta} \mj^* \mi M_{\Theta}^* = M_{\Theta} \mi \mj^* M_{\Theta}^*$, and hence
\[
 M_{z_i}^*P_{\cls}M_{z_i} M_{z_j}^*P_{\cls}M_{z_j} P_{\clq} = M_{z_i}^* M_{\Theta} \mi \mj^* M_{\Theta}^* M_{z_j} P_{\clq} = M_{\Theta} M_{\Theta}^* \pq = 0,
\]
which yields $(P_{\clq}-C_i^*C_i)(P_{\clq}-C_j^*C_j) = 0$. Thus we obtain
\[
(I_{\clq} - \ci^* \ci)(I_{\clq}- \cj^* \cj) = (P_{\clq}-C_i^*C_i)(P_{\clq}-C_j^*C_j)|_{\clq} = 0.
\]
Now we turn to the converse. By taking into account Lemma \ref{lemma: reduction}, what we have to show is that $X_{ij} = 0$. We now describe a multi-step reduction process that reduces this claim to Lemma \ref{lemma: 3 zero}. First observe that $\overline{\text{span}}\{ z^{k}\clq: k \in \mathbb{Z}^n_{+}\}$ reduces $M_{z_t}$ for all $t=1, \ldots, n$. Then there exists a closed subspace $\cle_1$ of $\cle$ such that 
\[
\overline{\text{span}}\{ z^{k}\clq: k \in \mathbb{Z}^n_{+}\} = H^2_{\cle_1}(\D^n).
\]
By setting $\cle_0 = \cle \ominus \cle_1$, it follows that
\[
H^2_{\cle}(\D^n) = \overline{\text{span}}\{ z^{k}\clq: k \in \mathbb{Z}^n_{+}\} \oplus H^2_{\cle_0}(\D^n).
\]
Since $\clq^\perp = \cls \supseteq H^2_{\cle_0}(\D^n)$, for each $f \in H^2_{\cle_0}(\D^n)$, we have $\pq M_{z_i}^* \ps f = \pq M_{z_i}^* f = 0$, as $H^2_{\cle_0}(\D^n)$ reduces $M_{z_i}$. This proves that
\[
X_{ij}|_{H^2_{\cle_0}(\D^n)} = 0.
\]
So we only need to check that $X_{ij}|_{\overline{\text{span}}\{ z^{k}\clq: k \in \mathbb{Z}^n_{+}\}} = 0$, or, equivalently
\[
X_{ij} M_z^l \pq = 0 \qquad (l \in \Z_+^n).
\]
Since $X_{ij} = \ps M_{z_i} \pq M_{z_j}^* \ps$ (see the definition of $X_{ij}$ in Lemma \ref{lemma: reduction}), we only need to consider $l \in \Z_+^n \setminus \{0\}$. Moreover, for each $f_0 \in  H^2_{\cle_0}(\D^n)$, since $M_{z_i}^* f_0 \in \cls$, it follows that
\[
\langle X_{ij} M_z^l f, f_0 \rangle =
\langle \ps M_{z_i} \pq M_{z_j}^* \ps M_z^l f, f_0 \rangle = \langle \pq M_{z_j}^* \ps M_z^l f, M_{z_i}^* f_0 \rangle = 0,
\]
for all $f \in \clq$ and $l \in \Z_+^n \setminus \{0\}$. Therefore, it suffices to prove that
\[
X_{ij} M_z^l \clq \perp M_z^k \clq \qquad (l \in \Z_+^n \setminus \{0\}, k \in \Z_+^n).
\]
Note that the case $k = 0$ is trivial since $\mbox{ran} X_{ij} \subseteq \cls$. Hence, we are reduced to showing that
\begin{equation}\label{eqn: final claim}
\pq M_z^{*k} X_{ij} M_z^l \pq = 0 \qquad (k, l \in \Z_+^n\setminus \{0\}).
\end{equation}
To prove this in full generality, we start with $k = e_i$ and $l = e_j$, where $e_i$ and $e_j$ are the multiindices with $1$ in the $i$- th and $j$-th slot, respectively, and zero elsewhere. In this case, we prove a little bit more, namely
\[
M_{z_i}^* X_{ij} M_{z_j} = 0.
\]
We proceed as follows: By applying \eqref{eqn: P_Q - CC*} twice we obtain
\[
\begin{split}
(\pq - C_i^* C_i) (\pq - C_j C_j^*) & = (\pq \mi^* \ps \mi \pq)(\pq \mj^* \ps \mj \pq)
\\
& = \pq (\mi^* \ps \mi) \pq (\mj^* \ps \mj) \pq
\end{split}
\]
Since $(\pq - C_i^* C_i) (\pq - C_j C_j^*) = 0$, by assumption, \eqref{eqn: Q reduces} implies that
\[
0 = \pq (\mi^* \ps \mi) \pq (\mj^* \ps \mj) \pq = (\mi^* \ps \mi)\pq (\mj^* \ps \mj) = \mi^* X_{ij} \mj,
\]
which proves the desired identity. In particular, \eqref{eqn: final claim} holds whenever $k, l \in \Z_+^n\setminus \{0\}$ and $k_i, l_j \neq 0$. Now let us consider the remaining cases: $k, l \in \Z_+^n\setminus \{0\}$, where

\noindent Case 1: $k_i = l_j = 0$,

\noindent Case 2: $k_i \neq 0$ and $l_j = 0$, and

\noindent Case 3: $k_i = 0$ and $l_j \neq 0$.

\noindent The first case simply follows from part $(1)$ of Lemma \ref{lemma: 3 zero}. For the remaining cases, we fix $k, l \in \Z_+^n \setminus \{0\}$. By \eqref{eqn: Q reduces} we have
\[
\pq M_z^{*k} \mi^* (\ps \mi \pq) = \pq M_z^{*k} (\mi^* \ps \mi) \pq = \pq M_z^{*k} \pq(\mi^* \ps \mi) \pq.
\]
Therefore, $\pq M_z^{*k} \mi^* (\ps \mi \pq) = (\pq M_z^{*k}) \pq \mi^* (\ps \mi \pq)$, from which it immediately follows that
\[
\pq M_z^{*k} \mi^* X_{ij} M_z^l \pq = (\pq M_z^{*k}) (\pq \mi^* X_{ij} M_z^l \pq),
\]
and similarly
\[
\pq M_z^{*k} X_{ij} \mj M_z^l \pq = (\pq M_z^{*k} X_{ij} \mj \pq) (M_z^l \pq).
\]
Then Case 2 and Case 3 follows from part $(2)$ and part $(3)$, respectively, of Lemma \ref{lemma: 3 zero}. This completes the proof that $\clq$ is a Beurling quotient module.
\end{proof}

\newsection{Isometric dilations}\label{sec: dilations}

This section is meant to complement the dilation theory of (a concrete class of) $n$-tuples of commuting contractions.

We begin with the definition of isometric dilations. Let $T=(T_{1},\ldots , T_{n})$ and $V = (V_1, \ldots, V_n)$ be commuting tuples of contractions and isometries on Hilbert spaces $\clh$ and $\clk$, respectively. We say that $V$ is an \textit{isometric dilation} of $T$ (or $T$ \textit{dilates} to $V$) if there exists an isometry $\Pi: \clh \raro \clk$ such that $\Pi T_i^* = V_i^* \Pi$ for all $i=1, \ldots, n$.

We will mostly restrict attention here to the case when $V$ is $(M_{z_1}, \ldots, M_{z_n})$ on $\he$ for some Hilbert space $\cle$. In fact, if $n=1$, then $T = (T)$ dilates to $M_z$ on $H^2_{\cle}(\D)$ for some Hilbert space $\cle$ if and only if $T$ is a pure contraction (recall again that an operator $X$ is pure if the sequence $\{X^{*m}\}_{m\geq 0}$ converges to $0$ in the strong operator topology). This deep result is due to Sz.-Nagy and C. Foias, and L. de Branges \cite{NF}. However, in sharp contrast, if $n=2$ ($n > 2$), then general $n$-tuples of pure commuting contractions do not dilate to $(M_{z_1}, M_{z_2})$ on vector-valued Hardy space over $\D^2$ (commuting tuples of isometries).

However, the multivariable situation is completely favorable in the case of Brehmer tuples (also popularly known as Szeg\"{o} tuples). Let $T=(T_{1},\ldots ,T_{n})$ be a commuting tuple of contractions on a Hilbert space $\clh$. We say that $T$ is \textit{Brehmer} if $T_i$ is pure for all $i=1, \ldots, n$, and
\[
\sum\limits_{F \subseteq \{1,\ldots ,n \}}(-1)^{|F|} T_{F}T^{*}_{F} \geq 0,
\]
where $|F|$ denotes the cardinality of $F$ and $T_{F}=\prod_{j\in F}T_{j}$ for all $F \subseteq \{1,\ldots ,n \}$. We set, by convention, that $T_{\emptyset}=I_{\mathcal{H}}$ and $|\emptyset|=0$. Given a Brehmer tuple $T$ on $\clh$, we define the defect operator and defect space of $T$ as
\[
D_{T^{*}}^{2} = \sum\limits_{F \subseteq \{1,\ldots ,n \}}(-1)^{|F|} T_{F}T^{*}_{F}, \text{~and~} \cld = \overline{\text{ran}}\,D_{T^*},
\]
respectively. The following is one of the most concrete multivariable dilation results \cite{CV, MV}:

\begin{thm}\label{canonical dilation}
If $T=(T_1,\ldots,T_n)$ is a Brehmer tuple on $\clh$, then $T$ dilates to $(M_{z_1}, \ldots, M_{z_n})$ on $H^2_{\cld}(\D^n)$.
\end{thm}

In particular, there exists an isometry $\Pi: \mathcal{H} \rightarrow {H}_{\mathcal{D}}^{2}(\mathbb{D}^{n})$ such that $\Pi T_i^* = M_{z_i}^* \Pi$ for all $i=1, \ldots, n$. Then $\clq := \Pi \clh$ is a quotient module of $H^2_{\cld}(\D^n)$, and hence
\[
T \cong (P_{\clq} M_{z_1}|_{\clq}, \ldots, P_{\clq} M_{z_n}|_{\clq}),
\]
on $\clq$. Note again that, if $n=1$, then $\clq$ is a Beurling quotient module, and hence
\[
T \cong P_{(\Theta H^2_{\cle_*}(\D))^\perp} M_z|_{(\Theta H^2_{\cle_*}(\D))^\perp},
\]
for some Hilbert space $\cle_*$ and inner function $\Theta \in H^\infty_{\clb(\cle_*, \cle)}(\D)$. This inner function $\Theta$ and the Beurling quotient module
\[
\clq_{\Theta} = (\Theta H^2_{\cle_*}(\D))^\perp,
\]
are popularly known as the \textit{characteristic function} of $T$ and the \textit{model space} corresponding to $T$, respectively \cite{NF}. In summary, pure contractions are unitarily equivalent to compressions of $M_z$ to model spaces.

We now study an analog of the above analytic model theorem for $n$-tuples of commuting contractions. First we set up some notation. Let $\cle$ and $\cle_*$ be Hilbert spaces, and let $\Theta \in H^\infty_{\clb(\cle_*, \cle)}(\D^n)$ be an inner function. Let us denote by $\clq_{\Theta} = \he \ominus \Theta H^2_{\cle_*}(\D^n)$ and $\cls_{\Theta} = \Theta H^2_{\cle_*}(\D^n)$ the Beurling quotient module and the Beurling submodule, respectively, corresponding to $\Theta$. We also define $T_{z_i, \Theta} = P_{\clq_{\Theta}} M_{z_i}|_{\clq_{\Theta}}$ for all $i=1, \ldots, n$, and set
\[
T_{\Theta} = (T_{z_1, \Theta}, \ldots, T_{z_n, \Theta}).
\]
One can now ask which $n$-tuples of commuting contractions are unitarily equivalent to $T_{\Theta}$ on Beurling quotient modules (or, model spaces) $\clq_{\Theta}$. The following result (a refinement of Theorem \ref{Char dc in terms of compre}) yields a complete answer to this question.

\begin{thm}\label{characteristic function}
Let $T=(T_1,\dots, T_n)$ be an $n$-tuple of commuting contractions on $\clh$. The following are equivalent.

\textup{(a)} $T \cong T_{\Theta}$ for some Beurling quotient module $\clq_{\Theta}$.

\textup{(b)} $T$ is a Brehmer tuple and $(I_{\clh}-T_i^*T_i)(I_{\clh}-T_j^*T_j)=0$ for all $i \neq j$.
\end{thm}

The proof directly follows from Theorem \ref{Char dc in terms of compre} and Theorem \ref{canonical dilation}.

\newsection{Factorizations and invariant subspaces}\label{sect: inv sub}

The main goal of this section is to classify factorizations of inner functions in terms of invariant subspaces of tuples of module operators. Our observation will also bring out a key difference between $n$-tuples of operators, $n > 1$, and single operators.

The structure of invariant subspaces of bounded linear operators has been traditionally related to the theory of (nontrivial) factorizations of one variable inner functions. For instance, the following result (see \cite[Chapter VI]{NF}, and more specifically \cite[Chapter 5, Proposition 1.21]{HB}) connects invariant subspaces of module (or model) operators with factorizations of the corresponding inner functions. Here we follow the same notation as in the discussion preceding Theorem \ref{characteristic function}.

\begin{thm}[Sz.-Nagy and Foias, and Bercovici] \label{fac inv subspace in D}
Let $\Theta\in H^\infty_{\clb(\cle_*, \cle)}(\D)$ be an inner function. Then $T_{\Theta}$ has an invariant subspace if and only if there exist a Hilbert space $\clf$ and inner functions $\Phi \in H^\infty_{\clb(\clf, \cle)}(\D)$ and $\Psi \in H^\infty_{\clb(\cle_*, \clf)}(\D)$ such that
\[
\Theta =\Phi \Psi.
\]
\end{thm}
In the above, the corresponding $T_{\Theta}$-invariant subspace is given by $\clm = \cls_{\Phi} \ominus \cls_{\Psi}$ \cite[Chapter 5, Proposition 1.21]{HB}. Here we are interested in the polydisc version of the above theorem. However, the following example shows that in the case when $n > 1$ the existence of joint invariant subspaces is not sufficient to ensure factorizations of inner functions.

\begin{ex}\label{example: not factorable}
Consider the submodule $\cls=\{f\in H^2(\D^2):f(0, 0)=0\}$ of $H^2(\D^2)$. Since
\[
\cls=z_1 (H^2(\D) \otimes \mathbb{C}) \oplus z_2  (\mathbb{C} \otimes H^2(\D)) \oplus z_1z_2H^2(\D^2),
\]
and $\cls$ is a reproducing kernel Hilbert space, the kernel function $k$ of $\cls$ is given by
\[
k(z,w) = \frac{z_1 \bar{w}_1}{1 - z_1 \bar{w}_1} + \frac{z_2 \bar{w}_2}{1 - z_2 \bar{w}_2} + z_1 z_2 \mathbb{S}(z, w) \bar{w}_1 \bar{w}_2 \qquad (z, w \in \D^2),
\]
where
\[
\mathbb{S}(z,w) = (1 - z_1 \bar{w}_1)^{-1} (1 - z_2 \bar{w}_2)^{-1} \qquad (z, w \in \D^2),
\]
is the Szeg\"{o} kernel of $\D^2$. A simple calculation shows that
\[
k(z,w) = (z_1(1 - z_2 \bar{w}_2)\bar{w}_1 + z_2 \bar{w}_2) \mathbb{S}(z, w)  \qquad (z, w \in \D^2).
\]
If possible, suppose that $\cls$ is a Beurling submodule, that is, $\cls = \theta H^2(\D^2)$ for some inner function $\theta \in H^\infty(\D^2)$. Then $k(z,w) = \theta(z) \overline{\theta(w)} \mathbb{S}(z, w)$, from which it immediately follows that
\[
\theta(z) \overline{\theta(w)} = z_1(1 - z_2 \bar{w}_2)\bar{w}_1 + z_2 \bar{w}_2 \qquad (z, w \in \D^2).
\]
Clearly, the left side is a positive definite function while the right side is not. This proves that $\cls$ is not a Beurling submodule. Now observe
\[
\vp(z)=\frac{2z_1z_2-z_1-z_2}{2-z_1-z_2} \qquad (z \in \D^2),
\]
defines an inner function in $H^{\infty}(\D^2)$. We have $\vp(0,0)=0$, and $\vp H^2(\D^2) \subsetneqq \cls \subsetneqq H^2(\D^2)$. Set $\clm = \cls \ominus \vp H^2(\D^2)$. Then $\clm$ is a non-trivial $(P_{\clq_{\vp}} M_{z_1}|_{\clq_{\vp}}, P_{\clq_{\vp}} M_{z_2}|_{\clq_{\vp}})$-invariant subspace of $\clq_{\vp}$, but $\vp$ is not factorable.
\end{ex}

The missing component in the polydisc analogue of Theorem \ref{fac inv subspace in D} will be determined in Theorem \ref{inv subsp in D^n}. But before that, we need a lemma. We will identify as usual $M_{z_i}$ on $H^2_{\cle}(\D^n)$ with $M_{z_i} \otimes I_{\cle}$ on $H^2(\D^n) \otimes \cle$, $i=1, \ldots, n$, and write
\[
M_z \otimes I_{\cle} = (M_{z_1} \otimes I_{\cle}, \ldots, M_{z_n} \otimes I_{\cle}).
\]
We know, for each $i=1, \ldots, n$, that
\[
M_{z_i}^* (\mathbb{S}(\cdot,w) \otimes \eta) = \bar{w}_i (\mathbb{S}(\cdot,w) \otimes \eta),
\]
and hence $M_{z_i} M_{z_i}^* (\mathbb{S}(\cdot,w) \otimes \eta) = z_i \bar{w}_i (\mathbb{S}(\cdot,w)$ for all $w \in \D^n$ and $\eta \in \cle$. It is now easy to see that $D^2_{M_z \otimes I_{\cle}} = P_{\mathbb{C}} \otimes I_{\cle}$ (see the definition of defect operators in Section \ref{sec: dilations}), where $P_{\mathbb{C}}$ denotes the orthogonal projection of $H^2(\D^n)$ onto the one-dimensional subspace of constant functions.

\begin{lem}\label{unitary constant}
Let $\Theta\in H^{\infty}_{\clb(\cle_*,\cle)}(\D^n)$ be an inner function. If $\Theta H^2_{\cle_*}(\D^n)=H^2_{\cle}(\D^n)$, then $ \Theta $ is an unitary constant.
\end{lem}

\begin{proof} Since, by hypothesis, $M_{\Theta}: H^2_{\cle_*}(\D^n)\rightarrow H^2_{\cle}(\D^n)$ is unitary and $(M_{z_i} \otimes I_{\cle}) M_{\Theta} = M_{\Theta} (M_{z_i} \otimes I_{\cle_*})$, it follows that $(M_{z_i}^* \otimes I_{\cle}) M_{\Theta} = M_{\Theta} (M_{z_i}^* \otimes I_{\cle_*})$ for all $i=1,\ldots,n$. Then
\[
D^2_{M_z \otimes I_{\cle}} M_{\Theta} = M_{\Theta} D^2_{M_z \otimes I_{\cle_*}},
\]
and hence
\[
(P_{\mathbb{C}} \otimes I_{\cle}) M_{\Theta} = M_{\Theta} (P_{\mathbb{C}} \otimes I_{\cle_*}).
\]
Thus, for any $\eta \in \cle_*$, we have $\Theta(z) \eta = \Theta(0) \eta$, $z \in \D^n$, that is, $\Theta \equiv \Theta(0)$ is a constant function. This completes the proof of the lemma.
\end{proof}

The $n=1$ case of the above lemma can be found in \cite[Chapter 5, Proposition 1.17]{HB}. Moreover, the present proof is slightly simpler.

We are now ready for the polydisc analog of Theorem \ref{fac inv subspace in D}. We will use the same notation as in the discussion preceding Theorem \ref{characteristic function}.

\begin{thm}\label{inv subsp in D^n}
Let $\Theta \in H^{\infty}_{\clb(\cle_*, \cle)}(\D^n)$ be an inner function. The following are equivalent.
\begin{enumerate}
\item There exist a Hilbert space $\clf$ and inner functions $\Psi$ and $\Phi$ in  $H^{\infty}_{\clb(\cle_*, \clf)}(\D^n)$ and $H^{\infty}_{\clb(\clf, \cle)}(\D^n )$, respectively, such that $\Theta = \Phi \Psi$.

\item There exists a $T_{\Theta}$-invariant subspace $\clm \subseteq \clq_{\Theta}$ such that $\clm\oplus \cls_{\Theta}$ is a Beurling submodule of $H^2_{\cle}(\D^n)$.

\item There exists a $T_{\Theta}$-invariant subspace $\clm \subseteq \clq_{\Theta}$ such that
\[
(I-C_i^*C_i)(I-C_j^*C_j)=0 \qquad (i \neq j),
\]	
where $C_s = P_{\clq_{\Theta}\ominus \clm}T_{z_s, \Theta}|_{\clq_{\Theta}\ominus\clm}$ for all $s = 1,\dots, n$.
\end{enumerate}
\end{thm}
\begin{proof}
$(1)\Rightarrow (2)$: Since $M_\Theta = M_\Phi M_\Psi$, we have $\Theta H^2_{\cle_*}(\D^n) \subseteq \Phi H^2_{\clf}(\D^n)$. Define
\[
\clm := \Phi H^2_{\clf}(\D^n) \ominus \Theta H^2_{\cle_*}(\D^n) = \cls_\Phi \ominus \cls_\Theta.
\]
Clearly, $\clm$ is a closed subspace of $\clq_{\Theta}$. Also note that
\[
\clq_{\Theta} \ominus \clm = (\he \ominus \Theta H^2_{\cle_*}(\D^n)) \ominus (\Phi H^2_{\clf}(\D^n) \ominus \Theta H^2_{\cle_*}(\D^n)),
\]
and hence, $\clq_{\Theta} \ominus \clm = \clq_{\Phi}$. Since $T_{z_i, \Theta}^* = \mi^*|_{\clq_\Theta}$ and $\clq_{\Phi} \subseteq \clq_{\Theta}$, it follows that $\clq_{\Phi}$ is $T_{z_i, \Theta}^*$-invariant for all $i=1,\dots, n.$ Consequently, $\clm$ is a $T_\Theta$-invariant subspace. For the second part, observe that
\[
\clm\oplus \cls_{\Theta} = \cls_{\Phi},
\]
is a Beurling submodule of $H^2_{\cle}(\D^n)$.

\NI {$(2)\Rightarrow (1)$}: Let $\clm$ is a $T_{\Theta}$-invariant subspace of $\clq_{\Theta}$ and suppose $\clm\oplus \cls_{\Theta}$ is a Beurling submodule of  $H^2_{\cle}(\D^n)$. Then (see the discussion preceding Lemma \ref{lemma: reduction}) there exist a Hilbert space $\clf$ and an inner function $\Phi \in H^{\infty}_{\clb(\clf, \cle)}(\D^n)$ such that
\[
\clm\oplus \cls_{\Theta}=\Phi H^2_{\clf}(\D^n).
\]
In particular, $\Theta H^2_{\cle}(\D^n)\subseteq \Phi H^2_{\clf}(\D^n)$, and hence, by Douglas's range and inclusion theorem, there exists a contraction
$ X: H^2_{\cle_*}(\D^n) \to H^2_{\clf}(\D^n)$ such that $M_{\Theta}=M_{\Phi}X$. But now, since $M_{\Phi}$ is an isometry and
\[
M_{\Phi}XM_{z_i} = M_{\Theta}M_{z_i} = M_{z_i}M_{\Theta}
= M_{z_i}M_{\Phi}X = M_{\Phi}M_{z_i}X,
\]
we find $XM_{z_i}=M_{z_i}X$ for all $i=1, \ldots,n$. Then there exists $\Psi \in H^{\infty}_{\clb(\cle_*, \clf)}(\D^n)$ such that $X = M_{\Psi}$. Finally, since $M_{\Theta}$ and $M_{\Phi}$ are isometries, we obtain
\[
\|f\|=\|M_{\Theta}f\|=\|M_{\Phi}M_{\Psi}f\|= \|M_{\Psi}f\| \qquad (f\in H^2_{\cle}(\D^n)),
\]
and hence, $M_{\Psi}$ is an isometry.

\NI $(1)\Rightarrow (3)$: As in the proof of $(1)\Rightarrow (2)$, if we set $\clm=\Phi H^2_{\clf}(\D^n)\ominus \Theta H^2_{\cle_*}(\D^n)$, then $\clq_{\Theta}\ominus \clm = \clq_{\Phi}$, which implies $C_s = P_{\clq_{\Phi}}M_{z_s}|_{\clq_{\Phi}}$ for all $s =1,\ldots, n$. Then the desired equality immediately follows from Theorem \ref{Char dc in terms of compre} applied to $(C_1, \ldots, C_n)$ on the Beurling quotient module $\clq_{\Phi}$.

\NI $(3)\Rightarrow (2)$: Since $T_{z_i, \Theta}^* = \mi^*|_{\clq_\Theta}$, $i=1, \ldots, n$, it follows that $\clq_{\Theta} \ominus \clm$ is a quotient module of $\he$. This says $\clq_{\Theta} \ominus \clm$ is a Beurling quotient module, taking into account the hypothesis and Theorem \ref{Char dc in terms of compre}. Finally, we observe
\[
\he \ominus (\clm \oplus \cls_\Theta) = \clq_\Theta \ominus \clm,
\]
which implies that $\clm \oplus \cls_\Theta$ is a Beurling submodule. This completes the proof of the theorem.
\end{proof}

It is now worthwhile to observe that the subspace $\clm \oplus \vp H^2(\D^2)$ in Example \ref{example: not factorable} is not a Beurling submodule.

Finally, let us concentrate on the trivial cases of the above theorem, namely, $\clm = \{0\}$ and $\clm = \clq_{\Theta}$. Recall that $\clm = \Phi H^2_{\clf}(\D^n)\ominus \Theta H^2_{\cle_*}(\D^n)$. Then $\clm = \{0\}$ if and only if $\Phi H^2_{\clf}(\D^n)= \Theta H^2_{\cle_*}(\D^n)$, which, since $\Theta = \Phi \Psi$, equivalent to $H^2_{\clf}(\D^n)= \Psi H^2_{\cle_*}(\D^n)$. By Lemma \ref{unitary constant}, the latter condition is equivalent to the condition that $\Psi$ is a unitary constant. For the second case, we note that $\clm = \clq_{\Theta}$ if and only if
\[
\Phi H^2_{\clf}(\D^n)\ominus \Theta H^2_{\cle_*}(\D^n) = \he \ominus \Theta H^2_{\cle_*}(\D^n),
\]
which is equivalent to $\he = \Phi H^2_{\clf}(\D^n)$. Therefore, we note, again by Lemma \ref{unitary constant}, that $\clm=\clq_{\Theta}$ if and only if $\Phi$ is a unitary constant. This proves that $\clm$ is a nontrivial $T_{\Theta}$-invariant subspace of $\clq_\Theta$ if and only if the inner functions $\Phi$ and $\Psi$ are not unitary constant.

In fact, something more can be said. We continue to use the setting and conclusion of Theorem \ref{inv subsp in D^n}.

\begin{cor}\label{cor: unique}
Let $\Theta \in H^{\infty}_{\clb(\cle_*, \cle)}(\D^n)$ be a nonconstant inner function.  Then the inner functions $\Phi$ and $\Psi$ are nonconstant if and only if the following holds:
\begin{enumerate}
\item $\clm$ is a nontrivial $T_{\Theta}$-invariant subspace of $\clq_{\Theta}$,
\item $\clm$ is not a Beurling submodule of $H^2_{\cle}(\D^n)$, and
\item $\clq_{\Theta}\ominus\clm$ does not reduce $M_z \otimes I_{\cle}$.
\end{enumerate}
\end{cor}
\begin{proof}
We have already seen that $\clm$ is a nontrivial subspace of $\clq_\Theta$ if and only if both the inner functions $\Phi$ and $\Psi$ are not unitary constant. In particular, if $\Phi$ and $\Psi$ are nonconstant, then $\clm$ is a nontrivial subspace of $\clq_{\Theta}$. Now suppose that $\clm$ is a Beurling submodule. Then there exist a Hilbert space $\clf_1$ and an inner function $\Phi_1 \in H^\infty_{\clb(\clf_1, \cle)}(\D^n)$ such that $\Phi H^2_{\clf}(\D^n)\ominus \Theta H^2_{\cle_*}(\D^n)=\Phi_1 H^2_{\clf_1}(\D^n)$. In particular, $\Phi_1 H^2_{\clf_1}(\D^n) \subseteq \Phi H^2_{\clf}(\D^n)$, which implies that $\Phi_1 = \Phi \Phi_2$ for some inner function $\Phi_2 \in  H^\infty_{\clb(\clf_1, \clf)}(\D^n)$. This yields $\Phi_1 H^2_{\clf_1}(\D^n)  = \Phi \Phi_2H^2_{\clf_1}(\D^n)$, and hence
\[
\Phi \Phi_2H^2_{\clf_1}(\D^n) = \Phi H^2_{\clf}(\D^n)\ominus \Phi \Psi H^2_{\cle_*}(\D^n)=\Phi\clq_{\Psi},
\]
from which we obtain $\clq_{\Psi}=\Phi_2H^2_{\clf_1}(\D^n)$. Thus $\clq_{\Psi}$, or equivalently, $\cls_\Psi$ reduces $M_z \otimes I_{\cle}$, which implies that $\Psi$ is a constant. This is a contradiction.

\noindent Finally, suppose towards a contradiction that $\clq_{\Theta}\ominus\clm$ reduces $M_z \otimes I_{\cle}$. Then
\[
\clm\oplus \cls_{\Theta} = H^2_{\cle}(\D^n)\ominus (\clq_{\Theta}\ominus\clm),
\]
also reduces $M_z \otimes I_{\cle}$. On the other hand, since $\clm\oplus \cls_{\Theta}=\Phi H^2_{\clf}(\D^n)$, it follows that $\Phi H^2_{\clf}(\D^n)=H^2_{\clf_2}(\D^n)$, and hence that $\Phi$ is a constant, which is a contradiction.

\NI Now we turn to the converse part. Suppose $\clm$ is a nontrivial $T_{\Theta}$-invariant subspace of $\clq_{\Theta}$. Since $\Theta = \Phi \Psi$ and $\Theta$ is nonconstant, both $\Phi$ and $\Psi$ can not be constant. Moreover, since $\clm$ is nontrivial, $\Phi$ and $\Psi$ cannot be unitary constants (see the discussion preceding the statement of the corollary). It remains to show that $\Phi$ and $\Psi$ cannot be constant isometry operators. First, let us assume that $\Phi \equiv V_1$ for some non-unitary isometry $V_1$ and that $\Psi$ is nonconstant. Then
\[
\clm\oplus  \Theta H^2_{\cle_*}(\D^n) =\Phi H^2_{\clf}(\D^n)=V_1 H^2_{\clf}(\D^n)=H^2_{V_1\clf}(\D^n),
\]
and hence $\clm\oplus \Theta H^2_{\cle_*}(\D^n)$ reduces $M_z \otimes I_{\cle}$, which is a contradiction. On the other hand, if $\Psi \equiv V_2$ and $\Phi$ is nonconstant, where $V_2$ is a non-unitary isometry, then
\[
\clm=\Phi H^2_{\clf}(\D^n)\ominus \Phi \Psi H^2_{\cle_*}(\D^n) = \Phi( H^2_{\clf}(\D^n)\ominus V_2 H^2_{\cle_*}(\D^n)) = \Phi H^2_{\clf\ominus V_2\cle^*}(\D^n),
\]
is a Beurling submodule of $H^2_{\cle}(\D^n)$, which is a contradiction. This completes the proof of the corollary.
\end{proof}

We refer the reader to the papers \cite{BL, GW, ZYL} and the survey \cite{Yang} for other results (mostly in two variables) on Beurling quotient modules.

\vspace{0.1in}

\noindent\textsf{Acknowledgement:}
The first named author acknowledges Indian Institute of Technology Bombay for warm hospitality. The research of the first named author is supported by the institute post-doctoral fellowship of IIT Bombay.
The second named author's  research work is supported by DST-INSPIRE Faculty Fellowship No. DST/INSPIRE/04/2015/001094. The fourth named author is supported in part by the Mathematical Research Impact Centric Support, MATRICS (MTR/2017/000522), and Core Research Grant (CRG/2019/000908), by SERB (DST), and NBHM (NBHM/R.P.64/2014), Government of India.

\end{document}